\documentclass[12pt,intlimits]{article}
\usepackage[a4paper,top=2cm,bottom=2cm,left=2cm,right=2cm,bindingoffset=5mm]{geometry}
\usepackage{amsmath}
\usepackage{amsthm}
\usepackage{amssymb}
\usepackage{graphicx}
\usepackage{epsfig}
\usepackage{hyperref}
\hypersetup{colorlinks=true,linkcolor=blue}
\newtheorem{theo}{Theorem}[section]
\newtheorem{lem}[theo]{Lemma}

\newtheorem{prop}[theo]{Proposition}

\theoremstyle{definition}
\newtheorem{defi}[theo]{Definition}
\theoremstyle{remark}
\newtheorem{remark}[theo]{Remark}

\newcommand{\be}{\begin{eqnarray}}
\newcommand{\ee}{\end{eqnarray}}
\newcommand{\bes}{\begin{eqnarray*}}
\newcommand{\ees}{\end{eqnarray*}}
\newcommand{\bi}{\begin{itemize}}
\newcommand{\ei}{\end{itemize}}
\newcommand{\ben}{\begin{enumerate}}
\newcommand{\een}{\end{enumerate}}

\newcommand{\La}{\mathcal{L}}

\newcommand{\R}{\mathbb{R}}
\newcommand{\N}{\mathbb{N}}
\newcommand{\C}{\mathbb{C}}
\newcommand{\X}{\mathbb{X}}
\newcommand{\G}{\mathcal{G}}

\newcommand{\de}{\mathrm {d}}
\def\Ext{{\hbox{\rm Ext}}}
\def\einschr{\hbox{\kern1pt\vrule height 6pt\vrule  width6pt height 0.4pt depth0pt\kern1pt}}

\newcommand{\Lm}{\mathfrak{L}}

\DeclareMathOperator{\dive}{div}

\DeclareMathOperator{\supp}{supp}
\DeclareMathOperator{\supess}{ess \sup}

\newcommand{\Dta}{\partial_t^\alpha}

\title{\bf Non-autonomous semilinear fractional evolution equations: well-posedness and ultracontractivity results  }

\date{}

\begin{document}
\maketitle

\centerline{\scshape Simone Creo and Maria Rosaria Lancia}
\medskip
{\footnotesize

 \centerline{Dipartimento di Scienze di Base e Applicate per l'Ingegneria, Sapienza Universit\`{a} di Roma,
}
   \centerline{via Antonio Scarpa 16, 00161 Roma, Italy.}
   \centerline{E-mail: simone.creo@uniroma1.it,\quad mariarosaria.lancia@uniroma1.it}
}

\vspace{1cm}

\begin{abstract}
\noindent We consider a time-fractional semilinear parabolic abstract Cauchy problem for a time-dependent sectorial operator $A(t)$ which satisfies the Acquistapace-Terreni conditions. We first prove local existence results for the mild solution of the problem at hand. Then we prove, under suitable assumptions on the initial datum, that the solution is also global in time. This is achieved by proving ultracontractivity estimates for the fractional evolution families associated with the operator $A(t)$.
\end{abstract}

\medskip



\bigskip

\section*{Introduction}
\setcounter{equation}{0}

In this paper we consider a semilinear non-autonomous fractional problem $(P)$:
\begin{equation}\notag
(P)
\begin{cases}
\Dta u(t)=A(t)u(t)+J(u(t))\quad &  t\in (0,T), \\
u(0)=u_0.
\end{cases}
\end{equation}
Here $\Dta$ is the fractional Caputo-type derivative, with $\alpha\in (0,1)$, $-A(t)\colon D(A(t))\subset X\to X$, $t\in [0,T]$, with $X$ Banach space,
are non-autonomous operators, that are the infinitesimal generators of analytic semigroups. Moreover we assume that the operators $A(t)$ satisfy the Acquistapace-Terreni conditions \cite{AT} (see also the seminal papers of Tanabe \cite{tanabe} and Sobolevskii \cite{sob}) (see Section \ref{sezAT}), where $D(A(t)):=\mathbb{X}$, (independent of $t$) for all $t\in [0,T]$ is a time invariant domain which is continuously embedded in $X$. $u_0$ is a given datum in a suitable functional space and $J$ is a Lipschitz mapping defined in $X$ satisfying suitable assumptions (see Section \ref{sec4}).
We prove well-posedness results for the mild solution and then we prove that the unique mild solution is global in time, under suitable assumptions on the initial datum.
Crucial tools for proving existence and uniqueness of the mild solution of problem $(P)$ are the proof of the ultracontractivity of the solution operators and a contraction argument in a suitable Banach space.

The literature on ATFE (autonomous time fractional equations) is wide, we refer e.g. to \cite{bazthesis,Baz2000,GLM,GKMRbook,kubyam} and to \cite{GWbook} and the references therein. Later, many results appeared on regularity and maximal regularity results for the solutions of linear and nonlinear fractional equations, see e.g. \cite{DG,DGJMAA}, \cite{clement} and \cite{Zacher} and the references listed there.

The study of NATFE (non-autonomous time fractional equations) is recent, the first results, in the linear case, are due to \cite{MonnPruss,Dong-Kim,bazthesis,El-Borai}. More recently, in \cite{hezhouFCAA}, under suitable hypotheses on the operators $A(t)$, the authors get an integral representation formula for the solution, which mimics the one in the autonomous case. The authors give a suitable notion of classical solution, they establish the solvability of the Cauchy problem and prove H\"older regularity results.
As to the non-autonomous fractional semilinear case, as far as we know the only result is due to \cite{hezhouBull}, where the authors consider almost sectorial operators.

Our interest in the study of problem $(P)$ is due to the fact that many problems in the framework of physical applications such as fractional diffusion advection equations \cite{Main2010}, as well as fractional order models of viscoelasticity \cite{Metzler2000} or  turbulent plasma \cite{DCCLy}, can be recast in the abstract form of problem $(P)$.

Problem $(P)$, in the case $\alpha=1$, with operators having time-dependent domains, has been also investigated in the framework of BVPs in domains with irregular boundaries or interfaces, possibly of fractal type, with different boundary conditions such as Dirichlet, Neumann, Robin or Wentzell boundary conditions. To our knowledge, the first results on non-autonomous semilinear Wentzell-type problems in irregular (or extension) domains are contained in \cite{CLADE} and in \cite{LVnonaut}, where an energy form approach is used. The generalization of these results to the case $0<\alpha<1$ is currently under investigation and it will be object of a forthcoming paper.

In the present paper we investigate the non-autonomous time fractional semilinear problem in a more general abstract setting.
Our aim here is to extend to the fractional non-autonomous case the ideas and methods of \cite{LVnonaut,CLADE} and \cite{Daners} under suitable hypotheses on $J(u)$. In order to use a fixed point argument in a suitable space of continuous functions, a crucial tool is to prove suitable mapping properties for $J(u)$. We remark that our fixed point argument does not involve neither fractional powers of the operator $A(t)$, nor interpolation spaces between the domain $D(A(t))$ and the space $X$. This approach turns out to be crucial in all those cases when the domain $D(A(t))$ is unknown. Instead, the proofs of our well-posedness results deeply rely on the ultracontractivity of the  operator valued functions $S_\alpha$ and $P_\alpha$ defined in Section \ref{sec2}. We stress the fact that the techniques used in \cite{LVnonaut} and \cite{CLADE} to prove the ultracontractivity property cannot be applied to the present case, since the problem is nonlocal in time. Here, the ultracontractivity property is obtained by adapting the techniques in \cite{GWbook} to the non-autonomous case.

The plan of the paper is the following.\\
In Section \ref{preliminari} we introduce the functional setting, the Caputo-type fractional time derivative and the non-autonomous operators $A(t)$.\\
In Section \ref{sec2} we introduce the two families of solution operators $S_\alpha$ and $P_\alpha$ and recall the representation formula for the solution of Problem $(P_L)$ (the linear version of Problem $(P)$).\\
In Section \ref{sec3} we prove the ultracontractivity of the two families of solution operators $S_\alpha$ and $P_\alpha$, see Theorem \ref{ultracontr}.\\
In Section \ref{sec4} we prove that the semilinear problem $(P_S)$ admits a unique mild solution. We prove that the mild solution is global in time under suitable assumptions on the initial datum.\\
In Section \ref{sec5} we present an application to the fractional heat problem $(P_H)$.

\section{Preliminaries}\label{preliminari}
\setcounter{equation}{0}

In this paper, by $C$ we denote possibly different positive constants. We give the dependence of constants on some parameters in parentheses. 

\subsection{Functional setting and fractional-in-time derivatives}\label{spazi funzionali}

Let $X$ be a Banach space endowed with a norm $\|\cdot\|$. By $C([a,b];X)$ we denote the space of continuous functions on a bounded interval $[a,b]\subset\R$ with values in $X$, and endow it with the supremum norm
\begin{equation}\notag
\|u\|_C:=\sup_{a\leq t\leq b} \|u(t)\|.
\end{equation}
Moreover, for $0<k\leq 1$, we denote by $C^k([a,b];X)$ the space
\begin{equation}\notag
C^k([a,b];X)=\left\{u\in C([a,b];X)\,:\,|u|_k:=\sup_{a\leq s<t\leq b} \frac{\|u(t)-u(s)\|}{|t-s|^k}<\infty\right\}
\end{equation}
endowed with the norm $\|u\|_{C^k}:=\|u\|_C+|u|_k$.

By $L^p(a,b;X)$, for $1\leq p\leq\infty$, we denote the Banach space of $X$-valued Bochner integrable functions defined on $[a,b]$ with the norms
\begin{equation}\notag
\begin{split}
&\|u\|_p:=\left(\int_a^b\|u(t)\|^p\,\de t\right)^\frac{1}{p},\quad\text{ for }1\leq p<\infty,\\
&\|u\|_\infty:=\supess_{a\leq t\leq b}\|u(t)\|,\quad\text{ for }p=\infty.
\end{split}
\end{equation}
Moreover, we denote by $\Lm(X,Y)$ the space of linear and continuous operators from a Banach space $X$ to a Banach space $Y$. If $X=Y$, we simply denote this space by $\Lm(X)$.


We recall the notion of fractional-in-time derivatives in the sense of Riemann-Liouville and Caputo by using the notations of the monograph \cite{GWbook}.

Let $\alpha\in (0, 1)$. We define
\begin{equation}\notag
g_\alpha (t)=
\begin{cases}
\displaystyle\frac{t^{\alpha -1}}{\Gamma(\alpha)}\quad &\text{if } t >0,\\
0  &\text{if } t\leq 0,
\end{cases}
\end{equation}
where $\Gamma$ is the usual Gamma function. In the following we will use the next property: for $t>\tau\geq 0$ and $\alpha,\beta>0$,
\begin{equation}\label{propgamma1}
\int_\tau^t (t-s)^{\alpha-1}(s-\tau)^{\beta-1}\,\de s=\frac{\Gamma(\alpha)\Gamma(\beta)}{\Gamma(\alpha+\beta)}(t-\tau)^{\alpha+\beta-1}.
\end{equation}

\begin{definition}\label{2.1.1}
Let $Y$ be a Banach space, $0\leq a<t\leq b$ and let $f\in C([a,b];Y)$ be such that $g_{1-\alpha}\ast f \in W^{1,1} ((a,b); Y)$.
\begin{itemize}
	\item[i)] The \emph{Riemann-Liouville} fractional derivative of order $\alpha \in  (0,1)$  is defined as follows:
$$D^\alpha_t f(t):=\frac{\de}{\de t}(g_{1-\alpha}\ast f)(t)=\frac{\de}{\de t}\int_a^t g_{1-\alpha}(t-\tau) f(\tau)\,\de\tau,$$
for a.e. $t>a$.
	\item[ii)] The \emph{Caputo-type} fractional derivative of order $\alpha \in  (0,1)$ is defined as follows:
$$\partial^\alpha_t  f(t):= D^\alpha_t (f(t)-f (a)),$$
for a.e. $t>a$.
\end{itemize}
\end{definition}

We stress the fact that Definition \ref{2.1.1}-$ii)$ gives a weaker definition of (Caputo) fractional derivative with respect to the original one (see \cite{CAPUTO}), since $f$ is not assumed to be differentiable. Moreover, it holds that $\partial^\alpha_t  (c) = 0$ for every constant $c\in\R$. We refer to the book \cite{Dieth} for further details on fractional derivatives.

\medskip

Now we recall the definition of the Wright type function (see \cite[Formula (28)]{GLM}):
\begin{equation}\notag
\Phi_\alpha(z):=\sum_{n=0}^\infty\frac{(-z)^n}{n!\Gamma(-\alpha n+1-\alpha)},\quad 0<\alpha<1,\,z\in \mathbb{C}.
\end{equation}

From \cite[page 14]{bazthesis}, it follows that $\Phi_\alpha(t)$ is a probability density function, i.e.
\begin{equation}\notag
\Phi_\alpha(t)\geq 0\quad\text{if }t\geq 0,\quad\int_0^{+\infty}\Phi_\alpha(t)\,\de t=1.
\end{equation}
Moreover,
\begin{equation}\label{propgamma2}
\int_0^\infty t^\delta\Phi_\alpha(t)\,\de t=\frac{\Gamma(1+\delta)}{\Gamma(1+\alpha\delta)}\quad \text{ for } -1<\delta<\infty.
\end{equation}
For more properties about the Wright function, among the others we refer to \cite{bazthesis}, \cite{GLM}, \cite{wright}.

\subsection{The non-autonomous operator}\label{sezAT}


Let us consider non-autonomous operators $A(t)$, for $t\in[0,T]$, such that $-A(t)\,: D(A(t))\subset X\to X$ are the infinitesimal generators of analytic semigroups $T_t(\tau)$ for $\tau\geq 0$.

Let us recall some basic notions of spectral analysis. For $t\in[0,T]$, let $\rho(A(t))$ be the resolvent set of $A(t)$. For every $\lambda\in\rho(A(t))$, we define the resolvent of $A(t)$ as $R(\lambda;A(t)):=(\lambda I-A(t))^{-1}$. Moreover, we assume that the spectrum of $A(t)$, denoted by $\sigma(A(t))$, is contained in a sectorial open domain for some fixed angle $\omega\subset(0,\frac{\pi}{2})$, i.e.,
\begin{equation}\notag
\sigma(A(t))\subset\Sigma_\omega:=\{\lambda\in\C\,:\;|\arg\lambda|<\omega\}.
\end{equation}

Now we introduce the main assumptions on the non-autonomous operators $A(\cdot)$, which are the so-called Acquistapace-Terreni conditions \cite{AT} (see also \cite{tanabe,sob}):
\begin{itemize}
	\item[i)] the domain of $A(t)$ is dense in $X$ and is independent of $t$, i.e.,
	\begin{equation}\label{AT1}
	D(A(t))=\X\quad\text{ for every } t\in[0,T];
	\end{equation}
	\item[ii)] the resolvent of $A(t)$ satisfies the following estimate: there exists a constant $M\geq 1$ such that, for every $\lambda\notin\Sigma_\omega$ and $t\in[0,T]$,
	\begin{equation}\label{AT2}
	\|(\lambda I-A(t))^{-1}\|_{\Lm(X)}\leq \frac{M}{1+|\lambda|};
	\end{equation}
	\item[iii)] there exist constants $L>0$ and $\theta\in(0,1]$ such that
	\begin{equation}\label{AT3}
	\|(A(t)-A(s))A(\tau)^{-1}\|_{\Lm(X)}\leq L|t-s|^\theta, \quad\text{ for }t,s,\tau\in[0,T],
	\end{equation}
	i.e. the operator $A(t)A(\tau)^{-1}$ is H\"older continuous in $t$ with exponent $\theta$.
\end{itemize}

\noindent We remark that in \cite[Hypothesis II, page 6]{AT} condition \eqref{AT3} appears with $t=\tau$. But, it can be proved (see \cite[page 110]{friedman}) that Hypothesis II in \cite{AT} actually implies condition \eqref{AT3}.

In what follows, we shall use fractional powers of the operator $A(t)$. Hence, for $0<\nu<1$, we introduce the operator $A(t)^{-\nu}\in\Lm(X)$ as
\begin{equation}\notag
A(t)^{-\nu}:=\frac{\sin\pi\nu}{\pi}\int_0^\infty s^{-\nu}\left(sI+A(t)\right)^{-1}\,\de s.
\end{equation}
We define $A(t)^\nu:=\left(A(t)^{-\nu}\right)^{-1}$ with domain $\X^\nu:=D(A(t)^\nu)$, endowed with the graph norm. We point out that, for some fixed $0<\nu<1$, it holds that
\begin{equation}\notag
D(A(s))\hookrightarrow\X^\nu\quad\text{ for every } 0\leq s\leq T.
\end{equation}

\medskip

We now introduce two operators which will play a crucial role in the rest of the paper. For $0<\alpha<1$, we set
\begin{equation}\label{def phi e psi}
\begin{split}
&\phi_t(\tau):=\int_0^\infty\Phi_\alpha(z) T_t(\tau^\alpha z)\,\de z,\quad\tau\geq 0,\\
&\psi_t(\tau):=\alpha\tau^{\alpha-1}\int_0^\infty z\Phi_\alpha(z)T_t(\tau^\alpha z)\,\de z,\quad \tau>0.
\end{split}
\end{equation}

Due to the previous assumptions, the analytic semigroups $T_t(\tau)$, for $\tau\geq 0$, satisfy the inequalities
\begin{equation}\label{bdd prop}
\begin{split}
&\|T_t(\tau)\|_{\Lm(X)}\leq M\quad\text{ for }\tau\geq 0,\\[2mm]
&\|A(t)^\nu T_t(\tau)\|_{\Lm(X)}\leq C\tau^{-\nu}\quad\text{ for }\tau>0,\nu>0,
\end{split}
\end{equation}
where $C$ is a positive constant.

\section{Representation formula for the linear case}\label{sec2}
\setcounter{equation}{0}

Aim of this section is to give a notion of solution of the following problem
\begin{equation}\notag
	(P_L)
	\begin{cases}
	\Dta u(t)-A(t)u(t)=f(t) \quad &\text{ for every }t \in (0,T),	\\
		u(0) =u_0,					
		\end{cases}
\end{equation}
where $f$ and $u_0$ are given functions in suitable Banach spaces, and to provide a representation formula for the possibly unique solution $u$.

Let us consider the following Volterra integral equation in $\Lm(X)$:
\begin{equation}\label{volterra}
\Lambda(t,\tau)=\Psi(t,\tau)+\int_\tau^t\Theta(t,s)\Lambda(s,\tau)\,\de s,\quad 0\leq\tau<t\leq T,
\end{equation}
where $\Psi(t,\tau)$ and $\Theta(t,\tau)$ are continuous in the uniform operator topology on $\Lm(X)$ and satisfy the following estimates, for some fixed exponents $\gamma,\beta>0$ and positive constant $C$:
\begin{equation}\notag
\|\Psi(t,\tau)\|_{\Lm(X)}\leq C(t-\tau)^{\gamma-1},\quad\|\Theta(t,\tau)\|_{\Lm(X)}\leq C(t-\tau)^{\beta-1},\quad 0\leq\tau<t\leq T.
\end{equation}

We now construct a unique solution to \eqref{volterra} by an iterative method, according to \cite{hezhouFCAA}.

For $m\in\N\cup\{0\}$, we define by induction
\begin{equation}\label{iterativo}
\Lambda_0(t,\tau)=\Psi(t,\tau),\quad \Lambda_{m+1}(t,\tau)=\Psi(t,\tau)+\int_\tau^t\Theta(t,s)\Lambda_m(s,\tau)\,\de s.
\end{equation}
We set
\begin{equation}\label{somma}
\Lambda(t,\tau)=\sum_{m=0}^\infty\Lambda_m(t,\tau).
\end{equation}

\begin{lemma}[Lemma 6 in \cite{hezhouFCAA}] The operator $\Lambda(t,\tau)$ defined in \eqref{somma} is the unique solution of \eqref{volterra}. Moreover, it is continuous in the uniform operator topology on $\Lm(X)$ for $0\leq\tau\leq t-\varepsilon\leq T$ for every $\varepsilon>0$ and there exists a constant $C>0$ such that
\begin{equation}\notag
\|\Lambda(t,\tau)\|_{\Lm(X)}\leq C(t-\tau)^{\gamma-1}E_{\gamma,\beta}(C(t-\tau)),\quad 0\leq\tau<t\leq T,
\end{equation}
where $E_{\gamma,\beta}$ is the Mittag-Leffler function.
\end{lemma}

\medskip

\noindent For the definition and properties of the Mittag-Leffler functions, we refer to \cite{podlubny,zhou2014}.

\medskip

We introduce two families of operators
\begin{equation}\notag
\tilde{Q}(t,\tau):=-(A(t)-A(\tau))\phi_\tau(t-\tau)\quad\tilde{R}(t,\tau):=-(A(t)-A(\tau))\psi_\tau(t-\tau).
\end{equation}
Then (see Lemma 7 in \cite{hezhouFCAA}) $\tilde{Q}(t,\tau)$ and $\tilde{R}(t,\tau)$ are continuous in the uniform operator topology for $0\leq\tau\leq t-\varepsilon\leq T$ for every $\varepsilon>0$ and there exists a constant $C>0$ such that
\begin{equation}\notag
\|\tilde{Q}(t,\tau)\|_{\Lm(X)}\leq C(t-\tau)^{\bar\omega-1},\quad\|\tilde{R}(t,\tau)\|_{\Lm(X)}\leq C(t-\tau)^{\theta-1},
\end{equation}
for $0\leq\tau<t\leq T$, where $\bar\omega:=\theta-\alpha+1>0$.

We now consider two operator-valued integral equations
\begin{equation}\label{volterraQR}
\begin{split}
&\tilde{Q}(t,\tau)+\int_\tau^t\tilde{R}(t,s)Q(s,\tau)\,\de s=Q(t,\tau),\\
&\tilde{R}(t,\tau)+\int_\tau^t\tilde{R}(t,s)R(s,\tau)\,\de s=R(t,\tau).
\end{split}
\end{equation}

As before, we define inductively, for $m\in\N\cup\{0\}$, 
\begin{equation}\notag
\begin{split}
Q_0(t,\tau)=\tilde{Q}(t,\tau),\quad Q_{m+1}(t,\tau)=\tilde{Q}(t,\tau)+\int_\tau^t\tilde{R}(t,s)Q_m(s,\tau)\,\de s,\\
R_0(t,\tau)=\tilde{R}(t,\tau),\quad R_{m+1}(t,\tau)=\tilde{R}(t,\tau)+\int_\tau^t\tilde{R}(t,s)R_m(s,\tau)\,\de s,
\end{split}
\end{equation}
and we set
\begin{equation}\notag
Q(t,\tau)=\sum_{m=0}^\infty Q_m(t,\tau),\quad R(t,\tau)=\sum_{m=0}^\infty R_m(t,\tau).
\end{equation}

\begin{lemma}[Lemma 8 in \cite{hezhouFCAA}] The operators $Q(t,\tau)$ and $R(t,\tau)$ defined above are the unique solutions of \eqref{volterraQR} and they are continuous in the uniform operator topology on $\Lm(X)$ for $0\leq\tau\leq t-\varepsilon\leq T$ for $\varepsilon>0$. Moreover, there exists a constant $C>0$ such that for $0\leq\tau<t\leq T$
\begin{equation}\label{stimaQR}
\|Q(t,\tau)\|_{\Lm(X)}\leq C(t-\tau)^{\bar\omega-1},\quad\|R(t,\tau)\|_{\Lm(X)}\leq C(t-\tau)^{\theta-1},
\end{equation}
where $\bar\omega=\theta-\alpha+1$.
\end{lemma}

\begin{lemma}[Lemma 10 in \cite{hezhouFCAA}] Let $\nu\in(0,1)$ be fixed. For every $u\in\X^\nu$ and $0<\beta\leq\min\{\theta,\alpha\nu\}$, there exists a constant $C>0$ such that
\begin{equation}\notag
\|Q(t,\tau)u-Q(s,\tau)u\|\leq C(t-s)^\beta(s-\tau)^{-\alpha}\|u\|_{\X^\nu}
\end{equation}
for every $0\leq\tau<s\leq t\leq T$.

For every $0<\beta<\theta\leq 1$, there exists a constant $C>0$ such that
\begin{equation}\notag
\|R(t,\tau)-R(s,\tau)\|\leq C(t-s)^\beta(s-\tau)^{\theta-\beta-1}
\end{equation}
for every $0\leq\tau<s\leq t\leq T$.
\end{lemma}

\bigskip

Our aim is now to construct two solution operators and to show some of their properties. In order to do so, let $Q(t,\tau)$ and $R(t,\tau)$ be the unique solutions of \eqref{volterraQR} for $0\leq\tau<t\leq T$. We set
\begin{equation}\label{defUV}
U(t,\tau):=\int_\tau^t\psi_s(t-s)Q(s,\tau)\,\de s,\quad V(t,\tau):=\int_\tau^t \psi_s(t-s)R(s,\tau)\,\de s.
\end{equation}
We define the solution operators $S_\alpha(t,\tau)$ and $P_\alpha(t,\tau)$ as
\begin{equation}\label{defSP}
S_\alpha(t,\tau):=\phi_\tau(t-\tau)+U(t,\tau)\,\quad P_\alpha(t,\tau):=\psi_\tau(t-\tau)+V(t,\tau).
\end{equation}

\begin{lemma}[Lemma 11 in \cite{hezhouFCAA}] The solution operators $S_\alpha(t,\tau)$ and $P_\alpha(t,\tau)$ defined in \eqref{defSP} are strongly continuous on $X$ for $t\in[\tau+\varepsilon,T]$ for $\varepsilon>0$. Moreover, there exists a constant $C>0$ such that
\begin{equation}\notag
\|S_\alpha(t,\tau)\|_{\Lm(X)}\leq C \text{ for }t\in[\tau,T],\quad\|P_\alpha(t,\tau)\|_{\Lm(X)}\leq C(t-\tau)^{\alpha-1} \text{ for }t\in(\tau,T].
\end{equation}
\end{lemma}

We remark that, in the autonomous case, i.e. when $-A(t)\equiv -A$ is independent of $t$, then the solution operators $S_\alpha(t,\tau)$ and $P_\alpha(t,\tau)$ defined in \eqref{defSP} reduce to the \lq\lq usual" solution operators (known also as resolvent families in literature) $S_\alpha(t)$ and $P_\alpha(t)$ defined as
\begin{equation}\notag
S_\alpha (t):=\int_0^{+\infty} \Phi_\alpha(s) T(s t^\alpha)\,\de s,\quad P_\alpha (t):= \alpha t^{\alpha -1 }\int_0^{+\infty}  s\Phi_\alpha(s) T(s t^\alpha)\,\de s,
\end{equation}
where $T(t)$ is the analytic semigroup generated by $-A$. For more details, we refer to \cite{zhou2014}, see also \cite{GWbook}.

\bigskip

We now give the notion of classical solution of problem $(P_L)$.

\begin{definition}[see Definition 3 in \cite{hezhouFCAA}] A function $u$ is a \emph{classical solution} of problem $(P_L)$ if $u$ is continuous and it satisfies problem $(P)$ on $[0,T]$ with $\Dta u\in C((0,T];X)$ and $u\in\X\equiv D(A(t))$ for every $t\in (0,T]$.
\end{definition}

The following existence and uniqueness result holds. We refer to Theorems 1 and 2 in \cite{hezhouFCAA}.

\begin{theorem}\label{exun lineare} Let $f\in C^k([0,T];X)$ and $\nu\in(0,1)$. Let assumptions \eqref{AT1}-\eqref{AT2}-\eqref{AT3} hold. Then for every $u_0\in\X^\nu$ problem $(P_L)$ admits a classical solution $u$ given by
\begin{equation}\label{soluzioneclassica}
u(t)=S_{\alpha}(t,0)u_0+\int_0^t P_{\alpha}(t,\tau)f(\tau)\,\de\tau,\quad t\in[0,T].
\end{equation}
Moreover, if $u_0\in\X$, the solution is unique.
\end{theorem}

\section{Ultracontractivity property}\label{sec3}
\setcounter{equation}{0}

In this section we investigate the ultracontractivity of the solution operators associated to problem $(P)$ in the Lebesgue framework.

From now on, let $\Omega$ be a relatively compact Hausdorff space and $\sigma$ be a Radon measure supported on $\overline\Omega$. We denote by $L^p(\Omega)\equiv L^p(\Omega,\sigma)$ the Banach space of Lebesgue integrable functions, endowed with the norm
\begin{equation}\notag
\|u\|_{L^p(\Omega)}:=\left(\int_\Omega |u|^p\,\de\sigma\right)^\frac{1}{p}.
\end{equation}
In the case $p=2$, we denote by $(\cdot,\cdot)_2$ the scalar product on $L^2(\Omega)$.

In this and in the following section, we set $X=L^2(\Omega)$. We set $\Xi:=\{(t,\tau)\in(0,T)^2\,:\,\tau<t\}$.

\begin{theo}\label{ultracontr} Let us assume that for every $t\in[0,T]$ the semigroup $T_t(\tau)$ generated by $-A(t)$ is ultracontractive, in the sense that:
\begin{itemize}
	\item[i)] for every $1\leq p\leq q\leq\infty$, $T_t(\tau)$ maps $L^p(\Omega)$ into $L^q(\Omega)$;
	\item[ii)] there exist two constants $C>0$ and $\lambda_A>0$ independent of $p$ and $q$ such that, for every $t>0$,
	\begin{equation}\label{hp ultracontr}
	\|T_t(\tau)\|_{\Lm(L^p(\Omega),L^q(\Omega))}\leq C\tau^{-\lambda_A\left(\frac{1}{p}-\frac{1}{q}\right)}.
	\end{equation}
\end{itemize}
Then, for every $(t,\tau)\in\Xi$, the following properties hold:
\begin{itemize}
	\item[a)] if $\lambda_A\left(\frac{1}{p}-\frac{1}{q}\right)<1$, then $S_\alpha(t,\tau)$ is ultracontractive in the sense that there exists a constant $C_S>0$ independent of $t$ and $\tau$ such that
	\begin{equation}\label{stima ultracontrS}
	\|S_\alpha(t,\tau)\|_{\Lm(L^p(\Omega),L^q(\Omega))}\leq C_S (t-\tau)^{-\alpha\lambda_A\left(\frac{1}{p}-\frac{1}{q}\right)};
\end{equation}
	\item[b)] if $\lambda_A\left(\frac{1}{p}-\frac{1}{q}\right)<2$, then $P_\alpha(t,\tau)$ is ultracontractive in the sense that there exists a constant $C_P>0$ independent of $t$ and $\tau$ such that
	\begin{equation}\label{stima ultracontrP}
	\|(t-\tau)^{1-\alpha}P_\alpha(t,\tau)\|_{\Lm(L^p(\Omega),L^q(\Omega))}\leq C_P(t-\tau)^{-\alpha\lambda_A\left(\frac{1}{p}-\frac{1}{q}\right)}.
\end{equation}
\end{itemize}
\end{theo}


\begin{proof}
We adapt to our setting the proof of \cite[Proposition 2.2.2]{GWbook}. Let $(t,\tau)\in\Xi$, $1\leq p\leq q\leq\infty$ and $f\in L^p(\Omega)$.

We begin by proving point $a)$. Here we assume $\lambda_A\left(\frac{1}{p}-\frac{1}{q}\right)<1$.

From \eqref{defSP}, \eqref{def phi e psi} and \eqref{defUV}  we have that
\begin{equation}\notag
\begin{split}
&\|S_\alpha(t,\tau)f\|_{L^q(\Omega)}\leq\|\phi_\tau(t-\tau)f\|_{L^q(\Omega)}+\|U(t,\tau)f\|_{L^q(\Omega)}\\
&\leq\int_0^\infty\Phi_\alpha(s)\|T_\tau(s(t-\tau)^\alpha)f\|_{L^q(\Omega)}\,\de s+\int_\tau^t\|\psi_s(t-s)Q(s,\tau)f\|_{L^q(\Omega)}\,\de s=:A_1+B_1.
\end{split}
\end{equation}
We first estimate $A_1$. From \eqref{hp ultracontr} and \eqref{propgamma2} it holds
\begin{equation}\notag
\begin{split}
A_1&=\int_0^\infty\Phi_\alpha(s)\|T_\tau(s(t-\tau)^\alpha)f\|_{L^q(\Omega)}\,\de s\\[2mm]
&\leq C(t-\tau)^{-\alpha\lambda_A\left(\frac{1}{p}-\frac{1}{q}\right)}\|f\|_{L^p(\Omega)}\int_0^\infty\Phi_\alpha(s)s^{-\lambda_A\left(\frac{1}{p}-\frac{1}{q}\right)}\,\de s\\
&=C(t-\tau)^{-\alpha\lambda_A\left(\frac{1}{p}-\frac{1}{q}\right)}\|f\|_{L^p(\Omega)}\frac{\Gamma\left(1-\lambda_A\left(\frac{1}{p}-\frac{1}{q}\right)\right)}{\Gamma\left(1-\alpha\lambda_A\left(\frac{1}{p}-\frac{1}{q}\right)\right)}\\
&\equiv C(\Gamma)(t-\tau)^{-\alpha\lambda_A\left(\frac{1}{p}-\frac{1}{q}\right)}\|f\|_{L^p(\Omega)},
\end{split}
\end{equation}
where $C(\Gamma)$ is a suitable positive constant dependent on the Euler function.

We now estimate the term $B_1$. By using \eqref{stimaQR} and again \eqref{def phi e psi}, \eqref{hp ultracontr} and \eqref{propgamma2} we have
\begin{equation}\notag
\begin{split}
B_1&=\int_\tau^t\|\psi_s(t-s)Q(s,\tau)f\|_{L^q(\Omega)}\,\de s\\[2mm]
&\leq C\alpha\int_\tau^t\int_0^\infty \Phi_\alpha(z) z^{1-\lambda_A\left(\frac{1}{p}-\frac{1}{q}\right)} (t-s)^{\alpha-1-\alpha\lambda_A\left(\frac{1}{p}-\frac{1}{q}\right)}\|Q(s,\tau)f\|_{L^q(\Omega)}\,\de z\,\de s\\
&\leq C\alpha\|f\|_{L^p(\Omega)}\int_\tau^t\int_0^\infty \Phi_\alpha(z) z^{1-\lambda_A\left(\frac{1}{p}-\frac{1}{q}\right)} (t-s)^{\alpha-1-\alpha\lambda_A\left(\frac{1}{p}-\frac{1}{q}\right)}(s-\tau)^{\bar\omega-1}\,\de z\,\de s\\
&=C(\Gamma)C\alpha\|f\|_{L^p(\Omega)}\int_\tau^t (t-s)^{\alpha-1-\alpha\lambda_A\left(\frac{1}{p}-\frac{1}{q}\right)}(s-\tau)^{\bar\omega-1}\,\de s,
\end{split}
\end{equation}
where $C(\Gamma)$ is a suitable positive constant dependent on the Euler function which can be explicitly calculated using \eqref{propgamma2}. Using now \eqref{propgamma1}, we obtain that there exists another suitable positive constant, which we denote again by $C(\Gamma)$ and can be explicitly computed, such that
\begin{equation}\notag
B_1\leq C(\Gamma)C\alpha\|f\|_{L^p(\Omega)}(t-\tau)^{\alpha-\alpha\lambda_A\left(\frac{1}{p}-\frac{1}{q}\right)+\bar\omega-1}\equiv C\alpha\|f\|_{L^p(\Omega)}(t-\tau)^{\theta-\alpha\lambda_A\left(\frac{1}{p}-\frac{1}{q}\right)}.
\end{equation}
Putting together the estimates for $A_1$ and $B_1$, we obtain that there exists a suitable positive constant, which we denote by $C_S$, such that, recalling that $(t,\tau)\in\Xi$,
\begin{equation}\notag
\begin{split}
\|S_\alpha(t,\tau)f\|_{L^q(\Omega)}&\leq C(t-\tau)^{-\alpha\lambda_A\left(\frac{1}{p}-\frac{1}{q}\right)}\|f\|_{L^p(\Omega)}\left[1+\alpha(t-\tau)^\theta\right]\\
&\equiv C_S(t-\tau)^{-\alpha\lambda_A\left(\frac{1}{p}-\frac{1}{q}\right)}\|f\|_{L^p(\Omega)},
\end{split}
\end{equation}
hence \eqref{stima ultracontrS} is proved.

We now prove point $b)$. Here we assume $\lambda_A\left(\frac{1}{p}-\frac{1}{q}\right)<2$.

From \eqref{defSP}, \eqref{def phi e psi} and \eqref{defUV} we have that
\begin{equation}\notag
\begin{split}
&\|(t-\tau)^{1-\alpha}P_\alpha(t,\tau)f\|_{L^q(\Omega)}\leq\|(t-\tau)^{1-\alpha}\psi_\tau(t-\tau)f\|_{L^q(\Omega)}+\|(t-\tau)^{1-\alpha}V(t,\tau)f\|_{L^q(\Omega)}\\
&\leq\alpha\int_0^\infty s\Phi_\alpha(s)\|T_\tau(s(t-\tau)^\alpha)f\|_{L^q(\Omega)}\,\de s+\int_\tau^t(t-\tau)^{1-\alpha}\|\psi_s(t-s)R(s,\tau)f\|_{L^q(\Omega)}\,\de s\\
&=:A_2+B_2.
\end{split}
\end{equation}
As to the term $A_2$, as in the previous step, by using \eqref{hp ultracontr} and \eqref{propgamma2}, we obtain that there exists a suitable constant $C>0$ such that
\begin{equation}\notag
A_2=\alpha\int_0^\infty s\Phi_\alpha(s)\|T_\tau(s(t-\tau)^\alpha)f\|_{L^q(\Omega)}\,\de s\leq C(t-\tau)^{-\alpha\lambda_A\left(\frac{1}{p}-\frac{1}{q}\right)}\|f\|_{L^p(\Omega)}.
\end{equation}
As to $B_2$, by using \eqref{stimaQR} and again \eqref{def phi e psi}, \eqref{hp ultracontr} and \eqref{propgamma2} we have
\begin{equation}\notag
\begin{split}
B_2&=\int_\tau^t(t-\tau)^{1-\alpha}\|\psi_s(t-s)R(s,\tau)f\|_{L^q(\Omega)}\,\de s\\
&\leq C\alpha(t-\tau)^{1-\alpha}\int_\tau^t\int_0^\infty \Phi_\alpha(z) z^{1-\lambda_A\left(\frac{1}{p}-\frac{1}{q}\right)} (t-s)^{\alpha-1-\alpha\lambda_A\left(\frac{1}{p}-\frac{1}{q}\right)}\|R(s,\tau)f\|_{L^q(\Omega)}\,\de z\,\de s\\
&\leq C\alpha(t-\tau)^{1-\alpha}\|f\|_{L^p(\Omega)}\int_\tau^t\int_0^\infty \Phi_\alpha(z) z^{1-\lambda_A\left(\frac{1}{p}-\frac{1}{q}\right)} (t-s)^{\alpha-1-\alpha\lambda_A\left(\frac{1}{p}-\frac{1}{q}\right)}(s-\tau)^{\theta-1}\,\de z\,\de s\\
&=C(\Gamma)C\alpha(t-\tau)^{1-\alpha}\|f\|_{L^p(\Omega)}\int_\tau^t (t-s)^{\alpha-1-\alpha\lambda_A\left(\frac{1}{p}-\frac{1}{q}\right)}(s-\tau)^{\theta-1}\,\de s,
\end{split}
\end{equation}
where as before $C(\Gamma)$ is a suitable positive constant dependent on the Euler function which can be explicitly computed using \eqref{propgamma2}. After using \eqref{propgamma1} we obtain
\begin{equation}\notag
B_2\leq C(\Gamma)C\alpha(t-\tau)^{1-\alpha}\|f\|_{L^p(\Omega)}(t-\tau)^{\alpha-\alpha\lambda_A\left(\frac{1}{p}-\frac{1}{q}\right)+\theta-1}\equiv C\alpha\|f\|_{L^p(\Omega)}(t-\tau)^{\theta-\alpha\lambda_A\left(\frac{1}{p}-\frac{1}{q}\right)}.
\end{equation}
Putting together these estimates, as before we have that there exists a suitable constant $C_P>0$ such that, for $(t,\tau)\in\Xi$,
\begin{equation}\notag
\|(t-\tau)^{1-\alpha}P_\alpha(t,\tau)f\|_{L^q(\Omega)}\leq C_P(t-\tau)^{-\alpha\lambda_A\left(\frac{1}{p}-\frac{1}{q}\right)}\|f\|_{L^p(\Omega)},
\end{equation}
hence \eqref{stima ultracontrP} is proved.

\end{proof}

\section{The semilinear problem: local and global existence}\label{sec4}
\setcounter{equation}{0}

In this section we investigate the semilinear problem
\begin{equation}\notag
	(P_S)
	\begin{cases}
	\Dta u(t)-A(t)u(t)=J(u(t)) \quad &\text{ for every }t \in (0,T),	\\
		u(0) =u_0,			
		\end{cases}
\end{equation}
under suitable assumptions on the semilinear term $J$. We provide both local and global existence results.

We recall that also in this section $X=L^2(\Omega)$, where $\Omega$ is as in Section \ref{sec3}. From now on we take $u_0\in L^2(\Omega)$.

\subsection{Local existence}

We introduce a suitable notion of solution to problem $(P_S)$.

\begin{definition}[see Definition 6.1 in \cite{hezhouBull}] A function $u\in C((0,T];X)$ is a \emph{mild solution} of problem $(P_S)$ if it satisfies
\begin{equation}\label{soluzionemild}
u(t)=S_{\alpha}(t,0)u_0+\int_0^t P_{\alpha}(t,\tau)J(u(\tau))\,\de\tau,\quad t\in(0,T].
\end{equation}
\end{definition}

We assume that for every $t\in (0,T)$ $J$ is a mapping from $L^{2p}(\Omega)$ to $L^2(\Omega)$ for $p>1$ which is locally Lipschitz, i.e., it is Lipschitz on bounded sets in $L^{2p}(\Omega)$:
\begin{equation}\label{LIPJ}
\|J(u)-J(v)\|_{L^2(\Omega)}\leq\sl{l(r)}\|u-v\|_{L^{2p}(\Omega)}
\end{equation}
whenever $\|u\|_{L^{2p}(\Omega)}\leq r,\|v\|_{L^{2p}(\Omega)}\leq r$, where $\sl{l(r)}$ denotes the Lipschitz constant of $J$.
We also assume that $J(0)=0$. This assumption is not necessary in all that follows, but it simplifies the calculations (see \cite{weiss1}).

In order to prove the local existence theorem, we make the following assumption on the growth of $\sl{l(r)}$ when $r\to+\infty$,
\begin{equation}\label{crescita l}
\mbox{Let}\; a:=1-\alpha+\frac{\alpha\lambda_A}{2}\left(1-\frac{1}{p}\right); \quad \mbox{there exists}\,\,0<b<a+\alpha-1\,:\,\sl{l(r)}=
{\mathcal{O}}(r^\frac{1-a}{b}),
\end{equation}
where $\lambda_A$ is the ultracontractivity exponent of $T_t(\tau)$ given in \eqref{hp ultracontr}. We note that $0<a<1$ for $\lambda_A<2$ and $p>1$. 

Let $p>1$. Following the approach in Theorem 2 in \cite{weiss1} and adapting the proof of Theorem 5.1 in \cite{La-Ve3}, we have the following result.

\begin{theorem}\label{theoesloc}
Let condition \eqref{crescita l} hold. Let $\kappa>0$ be sufficiently small, $u_0\in L^2(\Omega)$ and
\begin{equation}\label{cnidato}
\limsup_{t\to 0^+}\|t^b S_\alpha(t,0)u_0\|_{L^{2p}(\Omega)}<\kappa.
\end{equation}
Then there exists a $\overline{T}>0$ and a unique mild solution 
\begin{equation}\notag
u\in C([0,\overline{T}];L^2(\Omega))\cap C((0,\overline{T}];L^{2p}(\Omega)), 
\end{equation}
with $u(0)=u_0$ and $\|t^b u(t)\|_{L^{2p}(\Omega)}<2\kappa$, satisfying, for every $t\in [0,\overline{T}]$,
\begin{equation}\label{rappint1}
u(t)= S_\alpha(t,0) u_0 +\int_0^t P_\alpha(t,\tau) J(u(\tau))\,\de\tau,
\end{equation}
with the integral being both an $L^2$-valued and an $L^{2p}$-valued Bochner integral.
\end{theorem}

\begin{proof} The proof is based on a contraction mapping argument on suitable spaces of continuous functions with values in Banach spaces. We adapt the proof of Theorem 5.1 in \cite{La-Ve3} to this functional setting, see also Theorem 5.2 in \cite{CLADE}. For the reader's convenience, we sketch it.

We choose $\Lambda>0$ such that $l(r)\leq\Lambda r^{\frac{1-a}{b}}$ for $r\geq 1$ and we fix $u_0\in L^2(\Omega)$ and $0<b<a+\alpha-1$. Let $Y$ be the complete metric space defined by
\begin{equation}\label{spazioY}
\begin{split}
Y=&\left\{u\in C([0,\overline{T}];L^2(\Omega))\cap C((0,\overline{T}];L^{2p}(\Omega))\,:\,u(0)=u_0,\right.\\[2mm]
&\left.\|t^b u(t)\|_{L^{2p}(\Omega)}<2\kappa \mbox{ for every} \;t\in [0,\overline{T}]\right\},
\end{split}
\end{equation}
equipped with the metric $$d(u,v)=\max\left\{\|u-v\|_{C([0,\overline{T}];L^2(\Omega))},\,\sup_{(0,\overline{T}]}t^b \|u(t)-v(t)\|_{L^{2p}(\Omega)}\right\}.$$
For $w\in Y$, let $$\mathcal{F}w(t)=S_\alpha(t,0)u_0 +\int_0^t P_\alpha(t,\tau) J(w(\tau))\,\de\tau.$$ Then obviously $\mathcal{F}w(0)=u_0$ and, by using arguments similar to those used in \cite[proof of Lemma 2.1]{weiss3}, see also Theorem 2 in \cite{weiss1},
we can prove that, for $w\in Y$, $\mathcal{F}w\in C([0,\overline{T}];L^2(\Omega))\cap C((0,\overline{T}];L^{2p}(\Omega))$. We continue by proving that
\begin{equation}\label{tbFu}
\limsup_{t\rightarrow 0^+}\|t^b\mathcal{F}w(t)\|_{L^{2p}(\Omega)}<2\kappa \;\mbox{ for every}\; t\in [0,\overline{T}].
\end{equation}
This amounts to prove that there exists $\overline T>0$ such that $\|t^b \mathcal{F}w(t)\|_{L^{2p}(\Omega)}\leq 2\kappa$ for every $t\in [0,\overline{T}]$. From \eqref{cnidato}, Theorem \ref{ultracontr} and \eqref{crescita l} , we have
\begin{equation}\notag
\begin{split}
\|t^b \mathcal{F}w(t)\|_{L^{2p}(\Omega)}&\leq\kappa + t^b\left\| \int_0^t P_\alpha(t,\tau) J(w(\tau))\,\de\tau\right\|_{L^{2p}(\Omega)}\\
&\leq \kappa+ t^b \int_0^t \|P_\alpha(t,\tau)\|_{\Lm(L^2(\Omega),L^{2p}(\Omega))} \|J(w(\tau))\|_{L^2(\Omega)}\,\de\tau\\[2mm]
&\leq\kappa+C_Pt^b\int_0^t (t-\tau)^{\alpha-1-\alpha\lambda_A\left(\frac{1}{p}-\frac{1}{2p}\right)}\sl{l(\|w(\tau)\|_{L^{2p}(\Omega)})}\|w(\tau)\|_{L^{2p}(\Omega)}\,\de\tau\\
&\leq\kappa+C_Pt^b\Lambda\int_0^t (t-\tau)^{-a}\|w(\tau)\|_{L^{2p}(\Omega)}^\frac{1-a+b}{b}\,\de\tau.
\end{split}
\end{equation}
Now, from the definition of $Y$, it follows that
\begin{equation}\notag
\begin{split}
\|t^b \mathcal{F}w(t)\|_{L^{2p}(\Omega)}&\leq\kappa+C_pt^b\Lambda(2\kappa)^{\frac{1-a+b}{b}} \int_0^t (t-\tau)^{-a} \tau^{a-1-b}\,\de\tau\\
&=\kappa+C_p\Lambda(2\kappa)^{\frac{1-a+b}{b}} \int_0^1 (1-z)^{-a} z^{a-1-b}\,\de z,
\end{split}
\end{equation}
where in the last equality we made a suitable change of variables in the integral. We point out that the integral on the right-hand side of the above inequality is finite. Hence, if we set $B:=\int_0^1 (1-z)^{-a} z^{a-1-b}\,\de z>0$ and we choose $\kappa$ such that $C_P\Lambda B(2\kappa)^{\frac{1-a+b}{b}}\leq\kappa$, i.e. $\kappa\leq (C_P\Lambda B2^{\frac{1-a+b}{b}})^{-\frac{b}{1-a}}$, \eqref{tbFu} is proved.

It remains to prove that, for a suitable choice of $\kappa$, $\mathcal{F}$ is a contraction on $Y$.

By applying Theorem \ref{ultracontr} with $p=q=2$, and using again \eqref{cnidato} and \eqref{crescita l}, we have that
\begin{equation}\notag
\begin{split}
&\|\mathcal{F}u(t)-\mathcal{F}v(t)\|_{L^2(\Omega)}\leq \int_0^t\|P_\alpha(t,\tau)\left[J(u(\tau))-J(v(\tau))\right]\|_{L^2(\Omega)}\,\de\tau\\[2mm]
&\leq C_P\int_0^t(t-\tau)^{\alpha-1}\|J(u(\tau))-J(v(\tau))\|_{L^2(\Omega)}\,\de\tau\\[2mm]
&\leq C_P\Lambda (2\kappa)^{\frac{1-a}{b}}\int_0^t (t-\tau)^{\alpha-1}\tau^{a-1}\|u(\tau)-v(\tau)\|_{L^{2p}(\Omega)}\,\de\tau\\[2mm]
&\leq C_P\Lambda(2\kappa)^{\frac{1-a}{b}} \|u-v\|_Y \int_0^t (t-\tau)^{\alpha-1} \tau^{a-1-b}\,\de\tau\\[2mm]
&= C_P\Lambda(2\kappa)^{\frac{1-a}{b}} \|u-v\|_Y\,t^{\alpha-1+a-b}\int_0^1 (1-z)^{\alpha-1} z^{a-1-b}\,\de z.
\end{split}
\end{equation}
Therefore, we obtain
$$\|\mathcal{F}u-\mathcal{F}v\|_{C([0,\overline{T}],L^2(\Omega))} \leq C_PA\Lambda(2\kappa)^{\frac{1-a}{b}}\,{\overline T}^{\alpha-1+a-b} \|u-v\|_Y,$$
where $A:=\int_0^1 (1-z)^{\alpha-1} z^{a-1-b}\,\de z>0$.

We now consider the norm ${\|t^b(\mathcal{F}u(t)-\mathcal{F}v(t))\|_{L^{2p}(\Omega)}}$. Proceeding similarly as in the proof of \eqref{tbFu}, it holds that
\begin{equation}\notag
\begin{split}
&\|t^b(\mathcal{F}u(t)-\mathcal{F}v(t))\|_{L^{2p}(\Omega)}\leq t^b\int_0^t\|P_\alpha(t,\tau)\left[J(u(\tau))-J(v(\tau))\right]\|_{L^{2p}(\Omega)}\,\de\tau\\[2mm]
&\leq t^b\int_0^t \|P_\alpha(t,\tau)\|_{\Lm(L^2(\Omega),L^{2p}(\Omega))}\|J(u(\tau))-J(v(\tau))\|_{L^2(\Omega)}\,\de\tau\\[2mm]
&\leq C_Pt^b\Lambda(2\kappa)^{\frac{1-a}{b}}\int_0^t (t-\tau)^{\alpha-1-\alpha\lambda_A\left(\frac{1}{p}-\frac{1}{2p}\right)}\tau^{a-1-b}\|\tau^b(u(\tau)-v(\tau))\|_{L^{2p}(\Omega)}\,\de\tau\\[2mm]
&\leq C_Pt^b\Lambda(2\kappa)^{\frac{1-a}{b}}\|u-v\|_{Y}\int_0^t (t-\tau)^{-a}\tau^{a-1-b}\,\de\tau=C_PB\Lambda(2\kappa)^{\frac{1-a}{b}}\|u-v\|_{Y},
\end{split}
\end{equation}
where we recall that $B:=\int_0^1 (1-z)^{-a} z^{a-1-b}\,\de z>0$.

Hence, we have that
$$\sup_{t\in(0,\overline T]}\|t^b(\mathcal{F}u(t)-\mathcal{F}v(t))\|_{L^{2p}(\Omega)}\leq C_PB\Lambda(2\kappa)^{\frac{1-a}{b}}\|u-v\|_{Y},$$
and, in order to prove that it is a contraction, it is sufficient to choose $\kappa$ such that $C_PA\Lambda(2\kappa)^{\frac{1-a}{b}}\,{\overline T}^{\alpha-1+a-b}< 1$ and $C_PB\Lambda(2\kappa)^{\frac{1-a}{b}}<1$.

\end{proof}

\begin{remark}\label{remteoesloc}
If $J(u)= |u|^{p-1} u$, then $\sl{l(r)}= {\mathcal{O}}(r^{p-1})$ when $r\rightarrow+\infty$. Thus condition \eqref{crescita l} is satisfied for $b=\frac{\alpha}{p-1}-\frac{\alpha\lambda_A}{2p}$ with $p>1+\frac{2}{\lambda_A}$.
\end{remark}

We now prove the following Proposition, which will play a crucial role in the global existence result of the following subsection. 

\begin{proposition}\label{theoregloc}
Let the assumptions of Theorem \ref{theoesloc} hold. Let also condition \eqref{crescita l} hold. Then, the solution $u(t)$ can be continuously extended in $L^{2p}(\Omega)$ to a maximal interval $(0,\overline T_{u_0})$  as a solution of \eqref{rappint1}, until $\|u(t)\|_{L^{2p}(\Omega)}<\infty$.
\end{proposition}

\begin{proof}
We follow \cite[Theorem 2 and Corollary 2.1-b)]{weiss1}. From the proof of Theorem \ref{theoesloc}, it turns out that the existence time $\overline T_{u_0}<\infty$ for the solution to \eqref{rappint1} is as long as $\|t^b S_\alpha(t,0)u_0\|_{L^{2p}(\Omega)}\leq\kappa$. Moreover, from Theorem \ref{theoesloc} it holds that the map $t\mapsto u(t)$ is continuous both in $L^2(\Omega)$ and $L^{2p}(\Omega)$ for every $t\in(0,\overline T_{u_0}]$. Hence, the maximal existence time $\overline T_{u_0}$ is the same in $L^2(\Omega)$ and $L^{2p}(\Omega)$ and $\|u(t)\|_{L^{2p}(\Omega)}\to\infty$ as $t\to \overline T_{u_0}$. 
\end{proof}


\subsection{Global existence}

We now give a sufficient condition on the initial datum for obtaining a global solution, by adapting Theorem 3 (b) in \cite{weiss2}.

\begin{theorem}\label{theoesglob}
Let condition \eqref{crescita l} hold. Let $q:=\frac{2\alpha\lambda_A p}{\alpha\lambda_A+2pb}$, $u_0\in L^q(\Omega)$ and $\|u_0\|_{L^q(\Omega)}$ be sufficiently small. Then there exists a function $u\in C([0,\infty); L^q(\Omega))$ which is a global solution of \eqref{rappint1}.
\end{theorem}

\begin{proof}
Since $q<2p$, from the ultracontractivity of $S_\alpha(t,\tau)$ and $P_\alpha(t,\tau)$ it follows that 
\begin{equation}\label{stima Sa 2p-q}
\|S_\alpha(t,\tau)\|_{\Lm(L^q(\Omega),L^{2p}(\Omega))}\leq C_S (t-\tau)^{-\alpha\lambda_A\left(\frac{1}{q}-\frac{1}{2p}\right)}= C_S (t-\tau)^{-b},
\end{equation}
and
$$\|(t-\tau)^{1-\alpha}P_\alpha(t,\tau)\|_{\Lm(L^q(\Omega),L^{2p}(\Omega))}\leq C_P (t-\tau)^{-\alpha\lambda_A\left(\frac{1}{q}-\frac{1}{2p}\right)}=C_P (t-\tau)^{-b}.$$
From \eqref{stima Sa 2p-q}, we have that
$$ \|t^b S_\alpha(t,0)u_0\|_{L^{2p}(\Omega)}\leq C_S \|u_0\|_{L^q(\Omega)},$$ hence, by choosing $\|u_0\|_{L^q(\Omega)}$ sufficiently small, from Theorem \ref{theoesloc} we have that there exists a local solution of \eqref{rappint1} $u \in C([0,\overline T];L^q(\Omega))$. Furthermore, from Theorem \ref{theoesloc} we also have that $u \in C((0,\overline T];L^{2p}(\Omega))$ and $\|t^b u(t)\|_{L^{2p}(\Omega)}\leq 2C_S\|u_0\|_{L^q(\Omega)}$.

From Proposition \ref{theoregloc}, if we prove that $\|u(t)\|_{L^{2p}(\Omega)}$ is bounded for every $t>0$, then $u(t)$ is a global solution. We will prove that $\|t^b u(t)\|_{L^{2p}(\Omega)}$ is bounded for every $t>0$, and we will use the notations of the proof of Theorem \ref{theoesloc}.

As before, we choose $\Lambda>0$ such that $l(r)\leq\Lambda r^{\frac{1-a}{b}}$ for $r\geq 1$. Then 
\begin{equation}\notag
\begin{split}
&\|t^b u(t)\|_{L^{2p}(\Omega)}\leq C_S\|u_0\|_{L^q(\Omega)}+
 t^b \int_0^t \|P_\alpha(t,\tau)\|_{\Lm(L^2(\Omega),L^{2p}(\Omega))}\|J(u(\tau))\|_{L^2(\Omega)}\,\de\tau\\[2mm]
&\leq C_S\|u_0\|_{L^q(\Omega,m)}+ C_P\Lambda\left(2C_S \|u_0\|_{L^q(\Omega)}\right)^{\frac{1-a}{b}} t^b \int_0^t (t-\tau)^{-a} \tau^{a-1-b}\|\tau^b u(\tau)\|_{L^{2p}(\Omega)}\,\de\tau\\[2mm]
&\leq C_S\|u_0\|_{L^q(\Omega)}+C_P\Lambda\left(2C_S\|u_0\|_{L^q(\Omega)}\right)^{\frac{1-a}{b}}\sup_{t\in [0,T]} \|t^b u(t)\|_{L^{2p}(\Omega)} \int_0^1 (1-\tau)^{-a} \tau^{a-1-b}\,\de\tau,
\end{split}
\end{equation}
where in the last inequality we made the same change of variables in the integral that we did in the proof of Theorem \ref{theoesloc}.

Let now $f(T)=\sup_{t\in [0,T]} \|t^b u(t)\|_{L^{2p}(\Omega)}$. Then $f(T)$ is a continuous nondecreasing function with $f(0)=0$ which satisfies
$$f(T)\leq C_S\|u_0\|_{L^q(\Omega)}+\left(2C_S\|u_0\|_{L^q(\Omega)}\right)^{\frac{1-a}{b}} \Lambda BC_P f(T),$$
where $B:=\int_0^1 (1-\tau)^{-a} \tau^{a-1-b}\,\de\tau>0$. If $C_S\|u_0\|_{L^q(\Omega)}\leq \epsilon$ and $2^{\frac{1-a+b}{b}}\Lambda BC_P\epsilon^{\frac{1-a}{b}}<1$, then $f(T)$ can never be equal to $2\epsilon$. If it could, we would have $2\epsilon\leq\epsilon+(2\epsilon)^{\frac{1-a+b}{b}}\Lambda BC_P$, i.e., $\epsilon\leq(2\epsilon)^{\frac{1-a+b}{b}} \Lambda BC_P$, which is false if $\epsilon>0$ is small enough.

This proves that, if $\|u_0\|_{L^q(\Omega)}$ is sufficiently small, then $\|t^b u(t)\|_{L^{2p}(\Omega)}$ remains bounded, hence the assertion follows.
\end{proof}

\section{Application: Fractional heat equation for an operator in non-divergence form}\label{sec5}
\setcounter{equation}{0}

In this section we present an application to Boundary Value Problems. From now on, let $\Omega\subset\R^N$ a bounded domain with smooth boundary. For $x=(x_1,\dots,x_N)\in\Omega$, we adopt the following notation:
\begin{equation}\notag
\partial_i:=\frac{\partial}{\partial x_i},\quad\partial_{i,j}:=\frac{\partial^2}{\partial x_i\partial x_j},\quad i,j=1,\dots,N.
\end{equation}


The problem we consider in this section is formally stated as follows: 
\begin{equation}\notag
(\tilde P_H)\begin{cases}
\Dta u(t,x)-B(t,x,\nabla)u(t,x)=|u|^{p-1}u &\text{in $(0,T]\times\Omega$,}\\[2mm]
u(x,t)=0\quad &\text{on $(0,T]\times\partial\Omega$},\\[2mm]
u(0,x)=u_0(x) &\text{in $\overline\Omega$},
\end{cases}
\end{equation}
where $B$ is defined as
\begin{equation}\label{def B}
B(t,x,\nabla):=\sum_{i,j=1}^N a_{i,j}(t,x)\partial_{ij}+\sum_{i=1}^N b_{i}(t,x)\partial_{i}+c(t,x)I.
\end{equation}
We assume that the matrix $[a_{i,j}]$ is symmetric and that $B(x,t,\nabla)$ is uniformly elliptic in $\Omega$, i.e. there exists a constant $C>0$ such that
\begin{equation}\label{hp B1}
\sum_{i,j=1}^N a_{i,j}(t,x)\xi_i\xi_j\geq C|\xi|^2
\end{equation}
for every $x\in\overline\Omega$, $t\in[0,T]$ and $\xi\in\R^N$.

We take the functions $a_{i,j}(t,x)$, $b_i(t,x)$ and $c(t,x)$ to be real-valued and continuous in $\overline\Omega$ (for every fixed $t\in[0,T]$). We also suppose that the coefficients $a_{i,j}(t,x)$, $b_i(t,x)$ and $c(t,x)$ are H\"older continuous in $t$, i.e. there exist a constant $C_H>0$ and $0<\theta\leq 1$ such that
\begin{equation}\label{hp B2}
\begin{split}
&|a_{i,j}(t,x)-a_{i,j}(s,x)|\leq C_H|t-s|^\theta\quad\text{for every }x\in\overline\Omega,0\leq s,t\leq T\text{ and }i,j=1,\dots,N,\\[2mm]
&|b_{i}(t,x)-b_{i}(s,x)|\leq C_H|t-s|^\theta\quad\text{for every }x\in\overline\Omega,0\leq s,t\leq T\text{ and }i=1,\dots,N,\\[2mm]
&|c(t,x)-c(s,x)|\leq C_H|t-s|^\theta\quad\text{for every }x\in\overline\Omega,0\leq s,t\leq T.
\end{split}
\end{equation}
We take the same constant $C_H$ and the same H\"older exponent $\theta$ for all the coefficients of the operator $B$ for the sake of simplicity.

We remark that, as stated in Remark \ref{remteoesloc}, the semilinear term $J(u)=|u|^{p-1}u$ satisfies the hypotheses of the previous section, by suitably choosing the exponent $p$.

Following \cite[Part 2, Section 9]{friedman} (see also \cite[Section 7.6]{pazy}), we associate to $B(t,x,\nabla)$ a family of linear operators $A(t)$ in $L^2(\Omega)$, for $t\in[0,T]$, such that $D(A(t))\equiv\X:=H^2(\Omega)\cap H^1_0(\Omega)$ and $B(t,x,\nabla)u(t,x)=A(t)u(t,x)$ for $u\in\X$.

From \cite[Lemma 6.1, Chapter 7]{pazy}, it follows that under hypotheses \eqref{hp B1} and \eqref{hp B2} the operator $A(t)$ associated with $B(t,x,\nabla)$ satisfies the Acquistapace-Terreni conditions \eqref{AT1}-\eqref{AT3}. Hence, if we consider the following abstract version of problem $(\tilde P_H)$ in $L^2(\Omega)$, i.e.
\begin{equation}\notag
(P_H)\begin{cases}
\Dta u(t)-A(t)u(t)= |u(t)|^{p-1}u(t) &\text{for every $t\in(0,T)$,}\\[2mm]
u(0)=u_0,
\end{cases}
\end{equation}
where $u_0$ satisfies the hypotheses of Theorem \ref{theoesloc}, then from Theorem \ref{theoesloc} we deduce that problem $(P_H)$ admits a unique (local) mild solution. Naturally, if $u_0$ satisfies the hypotheses of Theorem \ref{theoesglob}, then problem $(P_H)$ admits a unique mild solution which is global in time.


Moreover, $-A(t)$ is the generator of a semigroup $T_t(\tau)$ which is analytic, contractive and strongly continuous for every $t\in[0,T]$ (see e.g. \cite[Chapter XVII]{DL5}). In addition, $T_t(\tau)$ is ultracontractive (see \cite[Section 6]{ouhabaz}) and if $N>2$ it satisfies the following estimate: for $0<\tau<t<T$ and $1\leq p\leq q\leq\infty$
\begin{equation}\notag
\|T_t(\tau)\|_{\Lm(L^p(\Omega),L^q(\Omega))}\leq C\tau^{-\frac{N}{2}\left(\frac{1}{p}-\frac{1}{q}\right)},
\end{equation}
i.e., in the notations of Theorem \ref{ultracontr}, $\lambda_A=\frac{N}{2}$, while, if $N\leq 2$, the semigroup $T_t(\tau)$ is ultracontractive of exponent $\lambda_A>1$ arbitrary (see \cite{GWbook}).

Hence, for every $0<\tau<t<T$, if $\lambda_A\left(\frac{1}{p}-\frac{1}{q}\right)<1$, both $S_\alpha(t,\tau)$ and $P_\alpha(t,\tau)$ are ultracontractive in the sense of Theorem \ref{ultracontr}.

\vspace{1cm}

\noindent {\bf Acknowledgements.} The authors have been supported by the Gruppo Nazionale per l'Analisi Matematica, la Probabilit\`a e le loro
Applicazioni (GNAMPA) of the Istituto Nazionale di Alta Matematica (INdAM). The authors report there are no competing interests to declare. They also thank MUR for the support under the project PRIN 2022 -- 2022XZSAFN: \lq\lq Anomalous Phenomena on Regular and Irregular Domains: Approximating Complexity for the Applied Sciences" -- CUP B53D23009540006 (see the website \href{https://www.sbai.uniroma1.it/~mirko.dovidio/prinSite/index.html}{https://www.sbai.uniroma1.it/~mirko.dovidio/prinSite/index.html}).

\end{document}